\documentclass[12pt]{amsart}
\usepackage{amsmath,amscd}
\usepackage{amssymb}
\usepackage{amsxtra}
\usepackage{enumerate}
\usepackage[mathscr]{eucal}
\usepackage[margin=2.5 cm]{geometry}

\usepackage{hyperref}
\usepackage{amsmath}
\usepackage{amsthm}
\usepackage{amssymb}
\usepackage{tikz}
\usepackage{float}
\usepackage[hypcap=false]{caption}
\usepackage{caption}
\usepackage{pst-node}
\usepackage{tikz-cd}

\newtheorem{theorem}{Theorem}[section]
\newtheorem{lemma}[theorem]{Lemma}

\newtheorem{corollary}[theorem]{Corollary}
\newtheorem{definition}[theorem]{Definition}
\theoremstyle{definition}

\usepackage{hyperref}

\usepackage{graphicx}
\usepackage{subfig}
\graphicspath{ {./images/} }

\numberwithin{equation}{section}

\usetikzlibrary{decorations.pathreplacing,decorations.markings}
\newcounter{flipflop}
\tikzset{
on each straight segment/.style={
    decorate,
    decoration={
        show path construction,
        moveto code={},
        lineto code={
            \path [#1]
            (\tikzinputsegmentfirst) -- (\tikzinputsegmentlast);
        },
        curveto code={
            \path  (\tikzinputsegmentfirst)
            .. controls
            (\tikzinputsegmentsupporta) and (\tikzinputsegmentsupportb)
            ..
            (\tikzinputsegmentlast);
        },
        closepath code={
            \path 
            (\tikzinputsegmentfirst) -- (\tikzinputsegmentlast);
        },
    },
},
set flipflop/.code=\setcounter{flipflop}{#1},
on each other straight segment/.style={
    decorate,
    decoration={
        show path construction,
        moveto code={},
        lineto code={\stepcounter{flipflop}
            \path \ifodd\value{flipflop} [#1]\fi
            (\tikzinputsegmentfirst) -- (\tikzinputsegmentlast);
        },
        curveto code={
            \path  (\tikzinputsegmentfirst)
            .. controls
            (\tikzinputsegmentsupporta) and (\tikzinputsegmentsupportb)
            ..
            (\tikzinputsegmentlast);
        },
        closepath code={
            \path 
            (\tikzinputsegmentfirst) -- (\tikzinputsegmentlast);
        },
    },
},
mid arrow/.style={postaction={decorate,decoration={
            markings,
            mark=at position .5 with {\arrow[#1]{stealth}}
}}},
}

\begin{document}

\title[Dehn function for Palindromic Sub-Group of Automorphism Group]{The Dehn function for Palindromic Sub-Group of $Aut(F_n)$}

\author[K. Gongopadhyay]{KRISHNENDU GONGOPADHYAY}

\author[L. Kundu]{LOKENATH KUNDU}

\address{
Indian Institute of Science Education and Research (IISER) Mohali,
		Knowledge City, Sector 81, SAS Nagar, Punjab 140306, India}
	\email{krishnendu@iisermohali.ac.in}

\address{Indian Institute of Science Education and Research (IISER) Mohali,
		Knowledge City, Sector 81, SAS Nagar, Punjab 140306, India.}
	\email{lokenath@iisermohali.ac.in}

\makeatletter
\@namedef{subjclassname@2020}{\textup{2020} Mathematics Subject Classification}
\makeatother



\maketitle



\begin{abstract}
  In this paper, we prove that the Dehn function of the palindromic automorphism group $\Pi A(F_n)$ is exponential. 
\end{abstract}

\section{Introduction}
 Let $G~=~\langle ~ X~|~R~\rangle$ be a finitely generated group. The word problem for $G$ is the method of determining whether two words in the generators represent the same element. The Dehn function is about how hard it is to solve word problems in a finitely presented group. To gain a deeper understanding of the Dehn function see \cite{Riley, Riley1}. If a group $G$ acts cocompactly and properly on a simplicial complex $S$, then the Dehn function of $G$ is asymptotically equivalent to the function that provides the optimal upper bound on the area of the least area discs in $S.$ Here, the restriction is stated as a function of the length of the boundary of the disc. In \cite{Bridson} authors proved that the Dehn function is exponential for $Out(F_n)$ and $Aut(F_n)$ for $n\geq 3$. 
 
 The subgroup of $Aut(F_n)$, consisting of all automorphisms that transform the generator of $F_n$ into palindromic words of $F_n,$ is known as the palindromic automorphism group. The palindromic group was introduced by Collins in \cite{Collins}. Let $a_1,a_2,\dots,a_n$ be the $n$-generators of $F_n.$ Following the notion by the authors in \cite{Glover}, a reduced word $a_1^{l_1}a_2^{l_2}\dots a_n^{l_n}$ is called a palindrome if $a_1^{l_1}a_2^{l_2}\dots a_n^{l_n}=a_n^{l_n}a_{n-1}^{l_{n-1}}\dots a_1^{l_1}.$ An automorphism $\alpha~\in~ Aut(F_n)$ is said to be a palindromic if $\alpha(a_i)$ is a palindromic for all $i$. The set of all palindromic automorphisms forms a subgroup of $Aut(F_n)$, and we denote it by $\Pi A(F_n)$.  For more details on the palindromic automorphism group, see \cite{Glover,Collins,KG}.
 
In this article, we will determine the Dehn function of $\Pi A(F_n)$. We prove the following:
\begin{theorem}\label{palin}
    The Dehn function of $\Pi A(F_n)$ is exponential, for $n\geq 3$. 
\end{theorem}
After discussing preliminary results and constructions, we established two lemmas in Section 3 that play a crucial role in proving the $Theorem \ref{palin}$. We prove this theorem in Section 4. 
\subsection{Acknowledgements}
The authors acknowledge support from the SERB  core research grant CRG/2022/003680  during the course of this work.
\subsection{Competing Interest}
We declare that there are no competing interests associated with this manuscript.

\section{Preliminaries}
\noindent In this section, we establish certain notations and revisit essential background information that will be employed consistently in the subsequent discussions within this document.
\subsection{Dehn function of a group} Consider a finitely generated group $G$ with  presentation $G=\langle S~|~R \rangle$. So $G$ is isomorphic to the quotient group $F(S)/N,$ where $F(S)$ represents the free group generated by the set $S,$ and $N$ is a normal subgroup of the free group $F(S)$. The elements in $R$ consist of cyclically reduced words, and every generator in $N$ is a conjugate of an element in $R.$ If a reduced word $w$ in  $F(S)$ represents the identity in $G$, we say that $w$ is null-homotopic and it is denoted as $w\equiv 1$ in $G$. Furthermore, this implies that $w$ belongs to $N=\langle\langle R\rangle\rangle.$ \begin{definition}{\rm (Area of null-homotopic words)}
If $w$ is equivalent to $1$ in the group $G,$ then $w$ can be expressed as the product $\displaystyle \prod_{i=1}^d g_ir_i^{\pm 1} g_i^{-1},$ where each $r_i$ belongs to the set $R.$ The area of $w$ is defined by the expression: \begin{equation}
    Area (w)=\min \lbrace d~|~w=\displaystyle \prod_{i=1}^d g_ir_i^{\pm 1} g_i^{-1}\rbrace.
\end{equation}
\end{definition}
\noindent The Dehn function quantifies the complexity of the word problem for any given group, and its definition is as follows:
\begin{definition}\label{DoG}
The Dehn function for a group $G$, denoted as $\delta_{(G,S)}: \mathbb{N}\rightarrow \mathbb{N}$, is defined as:
    \begin{equation}
    \begin{aligned}
        \delta_{(G,S)}(n)=\sup \lbrace Area(w)~| ~|w|\leq n,~w \equiv 1 \rbrace.
    \end{aligned}
    \end{equation}
Here, $S$ is any generating set for $G$, and $|w|$ represents the word length of an element $w \in F(S)$.
\end{definition}
 The Dehn function $\delta_{(G,S)}(n)$ is independent of generating sets up to an equivalence relation for any group $G=\langle S~|~R\rangle$. Henceforth we will denote the Dehn function by $\delta(n)$ of any group $G$. To explore more about the Dehn function see \cite{Riley, Riley1}. 

 In this context, we will have the next two outcomes without proof which will be crucial for demonstrating the Theorem \ref{palin}. To find a detailed proof of the following results, see \cite{Bridson}.
  \begin{lemma} \label{2.3}
      If $A$ and $B$ are 1-connected simplicial complexes, $F:~ A \rightarrow ~B$ is a simplicial map, and $l$ is a loop in the 1-skeleton of $A,$ then $Area_A(l)~\geq ~ Area_B(F\circ l).$   
  \end{lemma}
   \begin{corollary} \label{2.4}
      Let $A,B$ and $C$ 1-connected simplicial complexes with simplicial maps \begin{equation*}
          A~\rightarrow ~B\rightarrow ~ C.
      \end{equation*}\label{cor}
      Let $l_n$ be a sequence of simplicial loops in $A$ whose length is bounded above by a linear function of $n,$ let $\overline{l_n}$ be the image loops in $C$ and let $\alpha(n)= Area_C(\overline{l_n}).$ Then the Dehn function of $B$ satisfies $\delta_B(n)\succeq \alpha (n).$  
  \end{corollary}
  
\subsection{Palindromic automorphism group}
The palindromic automorphism group $\Pi A(F_n)$ is generated by the following automorphisms:
\begin{itemize}
    \item Maps $A_{ij}=(a_i||a_j),~i\neq j,$ which maps $a_i\mapsto a_ja_ia_j$ and fix all other generators $a_k,~k\neq i$.
\item Maps $\sigma_{a_i}$ which maps $a_i\mapsto a_i^{-1}$ and fix all other generators $a_k.$
\item Maps corresponding to elements of the symmetric group $S_n$ which permute the 
    $a_1,a_2,\\ \dots,a_n$
 among themselves.
\end{itemize}
 The subgroup of $\Pi A(F_n)$ which is generated by $A_{ij}$ only is called elementary palindromic subgroup $E\Pi A(F_n).$ Moreover $\Pi A(F_n)= ~E\Pi A(F_n)\rtimes(\mathbb{Z}/2 ~ \wr ~ S_n).$ Now \begin{equation}
     E \Pi A(F_n)= \langle A_{ij}~|~ 1\leq~ i,j\leq ~n,~ [A_{ik},A_{jk}],~[A_{ik},A_{jl}], ~A_{ik}A_{jk}A_{ij}=A_{ij}A_{jk}A_{ik}^{-1}\rangle. 
 \end{equation} 
To explore more about the palindromic group see \cite{Glover}.
 \subsection{Simplicial complex associated to the $\Pi A(F_n)$.} In this section we briefly describe a simplicial complex $L_{\Pi A(F_n)}$ on which the group $\Pi A(F_n)$ acts with finite stabilizer and finite quotient. The space $L_{\Pi A(F_n)}$ was introduced by Glover and Jensen in \cite{Glover}.\\
 Let $X_n$ be the spine of auter space \cite{Glover, Hatcher} which is a contractible simplicial complex with a cocompact proper action by $Aut(F_n)$. Let $\sigma_n~\in ~Aut(F_n)$ which sends all the generators of $F_n$ to its inverse. So, $\langle~\sigma_n~\rangle~= \mathbb{Z}/ 2.$ Let $R_n$ be a graph consisting of one vertex and $n$-edges. Hence $\langle~\sigma_n~\rangle$ acts on $R_n$ by inverting each edge. It is a notable point that $\Pi A(F_n)~=~C_{Aut(F_n)}(\sigma_n).$ Let $X_n^{\sigma_n}$ be the fixed point subspace of $X_n$ corresponding to $\langle~\sigma_n~\rangle$. From \cite{Jensen} $C_{Aut(F_n)}(\sigma_n)$ acts on the contractible space $X_n^{\sigma_n}$ with finite stabilizer and finite quotient. Moreover $X_n^{\sigma_n}$ is $\langle~\sigma_n~\rangle$-equivariantly deformation retraction to $L_{\sigma_n}.$ The space $L_{\sigma_n}$ constructed from $X_n^{\sigma_n}$ by considering the essential marked grpahs only. We will denote $L_{\sigma_n~}$ by $L_{\Pi A(F_n)}.$ The vertices of $L_{\Pi A(F_n)}$ can be represented as essential marked graphs with a chosen based point, and two vertices are adjacent if one of them is obtained from the other by a forest collapse. 
 
 Recall that an edge $e$ of a $G$ graph $\Gamma$ is said to be inessential if it is in every maximal $G$-invariant subforest of $\Gamma$. A graph is said to be inessential if it contains at least one inessential edge. A graph is said to be essential if it is not inessential. For more details about $L_{\Pi A(F_n)}$, see \cite{Glover}, pp. 646-655.
 
 Since $\Pi A(F_n)$ acts on the contractible space $L_{\Pi A(F_n)}$ with finite stabilizers and finite quotient so the Dehn function of $\Pi A(F_n)$ is asymptotically equivalent to the function that provides the optimal upper bound on the area of the least area discs in $L_{\Pi A(F_n)}$, where the bound is expressed as a function of the length of the boundary of the disc.      
\subsection{ Simplicial complexes associated to $Out(F_n)$.}
  Let $K_n$ denote the spine of the outer space as defined in \cite{Culler}. The group $Out(F_n)$ acts on $K_n$, and the action is cocompact and proper.
  A marked graph is a finite metric graph $\Gamma$ together with a homotopy equivalence $g:~R_n \rightarrow \Gamma$, where $R_n$ is a fixed graph with one vertex and $n$-loops. In a marked graph, two vertices are adjacent if one can be obtained from the other by a forest collapse. A vertex of $K_n$ can be represented as a marked graph $(g, \Gamma)$ with degree of all vertices at least three. For details see \cite{Culler}.
  \subsection{Nielsen transformation and Nielsen graph}
  Let $S=\lbrace s_1,s_2,\dots,s_n\rbrace$ be an ordered set. The following transformations are known as elementary Nielsen transformations:
  \begin{enumerate}
      \item[N1.] Replace some $s_i$ by $s_i^{-1}.$
      \item[N2.] Replace $s_i$ by $s_is_k$ for some $i\neq k.$
      \item[N3.] Delete $s_i$ if $s_i=1.$
  \end{enumerate}
  In the above transformations, all $s_j$ will be unchanged if $i\neq j.$ The product of elementary Nielsen transformations is known as Nielsen transformation.\\
  Let $G$ be a finitely generated group and the minimal number of generators for $G$ is at least $k.$ The Nielsen graph $N_k(G)$ is defined as follows:
  \begin{itemize}
      \item The vertex set of $N_k(G)$ consists of tuple $(g_1,g_2,\dots,g_k)$ such that $G=\langle g_1,g_2,\dots,g_k\rangle.$
      \item Two vertices of $N_k(G)$ are connected by an edge if one of them is obtained from the other by an elementary Nielsen move.
  \end{itemize}
  To explore more deeply into the Nielsen transformations and Nielsen graphs, refer to \cite{Nielsen2, Nielsen1}.
\section{Natural maps between simplicial complexes of $Out(F_n)$ and $\Pi A(F_n)$.}
There is a forgetful map $\phi_n~:~L_{\Pi A(F_n)} ~\rightarrow ~ K_n$ which simply forgets the base point, essentiality of marked graphs; this is a simplicial map.\\
Let $m<n.$ We fix an ordered basis for $F_n$, identify $F_m$ with the subgroup generated by the first $m$-elements of the basis, and identify $\Pi A(F_m)$ with the subgroup of $\Pi A(F_n)$ that leaves $F_m<F_n$ invariant and fixes the $n-m$ basis elements.\\
First, there is an equivariant map \begin{equation}
    i:~ L_{\Pi A(F_m)} \rightarrow L_{\Pi A(F_n)} 
\end{equation}
which attaches a bouquet of $n-m$ circles essentially to the basepoint of each essential marked graph and marks them with the last $n-m$ basis elements of $F_n.$ This map is simplicial since a forest collapse has no effect on the bouquet of circles at the base point.\\
Secondly, there is a restriction map $\rho:~K_n\rightarrow K_m,$ which can be describe in the following way: 
The chosen embedding $F_m<F_n$ corresponds to choosing an $m$-petal subrose $R_m\subset R_n.$ A vertex in $K_n$ is given by a graph $\Gamma$  marked with a homotopy equivalence $g:~R_n ~\rightarrow ~ \Gamma,$ and the restriction of $g$ to $R_m$ lifts to a homotopy equivalence $\hat{g}:~ R_m~ \rightarrow~ \hat{\Gamma},$ where $\hat{\Gamma}$ is the covering space corresponding to $g_*(F_m).$ There is a canonical retraction $r$ of $\hat{\Gamma}$ onto its compact core, i.e. the smallest connected subgraph containing all nontrivial embedded loops in $\Gamma.$ Let $\hat{\Gamma_0}$ be the graph obtained by erasing all vertices of valence $2$ from the compact core and define $\rho(g,\Gamma)~=~(r~\circ ~\hat{g},\hat{\Gamma_0}).$ To gain a deeper understanding see \cite{Bridson}. We now present the following lemma without proof, for details of proof see \cite{Bridson},
\begin{lemma} \label{dia}
    For $m<n,$ the restriction map $\rho:~K_n\rightarrow K_m$ is simplicial.
\end{lemma}

\begin{lemma}
    For $m<n,$ the following diagram of simplicial maps commutes:

$$\begin{CD}
L_{\Pi A(F_m)} @>i>> L_{\Pi A(F_m)}\\
@VV\phi_mV @VV\phi_nV\\
K_m @<\rho<< K_n
\end{CD}$$   
\end{lemma}
\begin{proof}
    Let $(g,\Gamma;v)$ be an essential marked graph with a base point in $L_{\Pi A(F_m)}.$ Then the essential marked graph $i(g,\Gamma;v)$ with base point at $v$ is obtained by attaching $n-m$ loops essentially at the base point $v$ labelled by the elements $a_{m+1},\dots, a_n$ of our fixed basis for $F_n.$ Now $(g_n,\Gamma_n):~=~\phi_n\circ i(g,\Gamma;v)$ is obtained by forgetting the base point $v$ and the essentiality of $(g,\Gamma)$,  and the cover of $(g_n,\Gamma_n)$ corresponding to $F_m<F_n$ is obtained from a copy of $(g,\Gamma)$ (with labels) by attaching $2(n-m)$ trees. (These trees are obtained from the Cayley graph of $F_n$ as follows: one cuts at an edge labeled $a_i^{\epsilon},$ with $i=m+1,\dots,n$ and $\epsilon= \pm 1,$ takes one component of the result, and then attaches the hanging edge to the base point $v$ of $\Gamma.$) The effect of $\rho$ is to delete these trees.
\end{proof}
\section{Proof of the theorem \ref{palin}.}
\begin{proof}
In \cite{Hatcher} authors proved an exponential upper bound on the Dehn function of $Aut(F_3)$ and $Out(F_3)$ for all $n\geq 3.$ In \cite{Mosher} authors established an exponential lower bound by using their general results on quasi-retractions to reduce to the case $n=3.$ So it is enough to prove our Theorem \ref{palin} for $n=3.$\\
In light of the Corollary \ref{cor} and Lemma \ref{dia}, it suffices to exhibit a sequence of loops $l_p$ in the 1-skeleton of $L_{\Pi A(F_3)}$ whose lengths are bounded by a linear function of $p$ and whose filling area when projected to $K_3$ grows exponentially as a function of $p.$ Instead of loops we will consider words in $\Pi A(F_3)$, but standard quasi-isometric arguments show this is equivalent. Consider words of the form:

\begin{equation}
    W_p:=A_{ij}^{-p}A_{kj}^{-p}A_{ij}^p A_{ki} A_{ij}^p A_{kj}^{-p} A_{ij}^{-p} A_{ki}^{-1}.
\end{equation}
We will prove that the word $W_p$ is null-homotopic word in $\Pi A(F_3).$
\begin{equation}
    A_{ij}(a_i)=a_j~a_i~a_j,~ A_{ij}^{-1}(a_i)=a_j^{-1}~a_i~a_j^{-1},~ A_{ij}(a_k)=a_k~\forall~k\neq i.
\end{equation}
\begin{equation}
\begin{split}
W_p(a_k) & = A_{ij}^{-p}A_{kj}^{-p}A_{ij}^p A_{ki} A_{ij}^p A_{kj}^{-p} A_{ij}^{-p} A_{ki}^{-1}(a_k) \\
 & = A_{ij}^{-p}A_{kj}^{-p}A_{ij}^p A_{ki} A_{ij}^p A_{kj}^{-p} A_{ij}^{-p}(a_i^{-1}a_ka_i^{-1})\\
 & = A_{ij}^{-p}A_{kj}^{-p}A_{ij}^p A_{ki} A_{ij}^p A_{kj}^{-p} (a_i^{-1}a_ka_i^{-1})\\
 & = A_{ij}^{-p}A_{kj}^{-p}A_{ij}^p A_{ki} A_{ij}^p (a_i^{-1} a_j^{-p} a_k a_j^{-p} a_i^{-1})\\
 & = A_{ij}^{-p}A_{kj}^{-p}A_{ij}^p A_{ki}  (a_i^{-1} a_j^{-p} a_k a_j^{-p} a_i^{-1})\\
 & = A_{ij}^{-p}A_{kj}^{-p}A_{ij}^p   (a_i^{-1} a_j^{-p}  a_i a_k a_i a_j^{-p} a_i^{-1})\\
 & = A_{ij}^{-p}A_{kj}^{-p}   (a_j^{p}  a_k a_j^{p})\\
 & = A_{ij}^{-p}   (a_j^{p} a_j^{-p} a_k a_j^{-p} a_j^{p})\\
 & = A_{ij}^{-p}   (a_k)\\
 & = a_k
\end{split}
\end{equation}

\begin{equation}
\begin{split}
W_p(a_i) & = A_{ij}^{-p}A_{kj}^{-p}A_{ij}^p A_{ki} A_{ij}^p A_{kj}^{-p} A_{ij}^{-p} A_{ki}^{-1}(a_i) \\
 & = A_{ij}^{-p}A_{kj}^{-p}A_{ij}^p A_{ki} A_{ij}^p A_{kj}^{-p} (a_j^{-p}a_ia_j^{-p})\\
 & = A_{ij}^{-p}A_{kj}^{-p}A_{ij}^p A_{ki} (a_j^{-p} a_j^p a_i a_j^p a_j^{-p})\\
 & = A_{ij}^{-p}A_{kj}^{-p}(a_j^p a_i a_j^p)\\
 & = A_{ij}^{-p}(a_j^p a_i a_j^p)\\
 & = a_j^p a_j^{-p} a_i a_j^{-p} a_j^p\\
 & = a_i
\end{split}
\end{equation}

Similarly, we can compute $W_p(a_j)~=~a_j$. Therefore, $W_p~\equiv~1$, indicating $W_p$ represents a null-homotopic word in $\Pi A(F_3).$

We will interpret these words as loops. Note that $A_{ij}=\rho_{ij}\lambda_{ij}$ where $\lambda_{ij}$ and $\rho_{ij}$ are elementary transvection defined as follows: \begin{equation}
    \lambda_{ij}(a_i)=a_ja_i, ~~\rho_{ij}(a_i)=a_ia_j.
\end{equation}
So $W_p$ is the product of $12p+4$ elementary transvections. There is a connected subcomplex of the 1-skeleton of $L_{\Pi A(F_3)}$ spanned by roses and Nielsen graphs. We say roses are adjacent if they have distance $2$ in this graph.

Let $I \in L_{\Pi A(F_3)}$ be the rose marked by the identity map $R_3\rightarrow R_3.$ Each elementary transvection $\tau$ moves $I$ to an adjacent rose $\tau I,$ which is connected to $I$ by a Nielsen graph $N_{\tau}$ in $L_{\Pi A(F_3)}$. A composition $\tau_1\dots \tau_s$ of elementary transvections gives a path through adjacent roses $I, \tau _1 I, \dots, \tau_1\dots \tau_s I;$ the Nielsen graph connecting $\sigma I$ to $\sigma \tau I$ is $\sigma N_{\tau}.$ Thus the word $W_p$ corresponds to a loop $l_p$ in $L_{\Pi A(F_3)}.$ \cite[Theorem A]{Vogtmann} which states "$Aut(F_3)$ and $Out(F_3)$ satisfy a strictly exponential isoperimetric inequality." Therefore there is an exponential lower bound on the filling area of $\phi_3 \circ l_p$ in $K_3.$ Using Lemma \ref{2.3} and Corollary \ref{2.4}, we will have the exponential lower bound of the filling area of $l_p$ in $L_{\Pi A(F_3)}$. Hence the Dehn function of the palindromic automorphism group $\Pi A(F_n)$ is exponential. 
\end{proof}

%
%


\begin{thebibliography}{19}
\bibitem{Vogtmann}
Martin R. Bridson, Karen Vogtmann: On the geometry of the automorphism group of a free group. Bull. Math Londres. Soc. 27 (1995), 544-552.

\bibitem {Bridson}
Martin R. Bridson, Karen Vogtmann: The Dehn function of $Out(F_n)$ and $Aut(F_n)$. Annales De L\'Institut Fourier, Gernoble, 62, 5 (2012), 1811-1817. 

\bibitem{Culler}
M. Culler, K. Vogtmann: Moduli graphs and automorphisms of free groups. Invent. Math. 84 (1986), no. 1. 91-119. 

\bibitem {Riley}
T Riley: Dehn functions.  in: Office hours with a geometric group theorist, 146–175.
Princeton University Press, Princeton, NJ, 2017

\bibitem {Riley1}
T Riley: What is Dehn function? \url {https://pi.math.cornell.edu/~riley/papers/What_is_a_Dehn_function/What_is_a_Dehn_function.pdf}

\bibitem {Glover}
Henry H. Glover, Craig A. Jensen: Geometry for palindromic automorphism groups of free groups. Comment. Math. Helv. 75(2000) 644-667.


\bibitem{Hatcher} A Hatcher, K Vogtmann: Isoperimetric inequalities for automorphism groups of free groups. Pacific j. Math. 173 (1996), 2, 425-441.

\bibitem{GRO} M. Gromov, Asymptotic invariants of infinite groups, in: Geometric Group Theory.
Vol. 2, London Math. Soc. Lecture Note Ser. 182, Cambridge University, Cambridge (1993), 1–295, (Sussex) 1991.
\bibitem{Mosher} M. Handel, L. Mosher: Lipschitz retraction and distortion for subgroups of $Out(F_n)$. \hyperlink{arXiv:1009.5018}{arXiv:1009.5018}, 2010.
\bibitem{Nielsen1} A Myropolska: Nielsen equivalence in Gupta-Sidiki groups. Proceedings of American Mathematical Society. Volume 145, Number 8, August 2017, Pages 3331–3342. \hyperlink{, Pages 3331–3342
http://dx.doi.org/10.1090/proc/13612}{, Pages 3331–3342
http://dx.doi.org/10.1090/proc/13612}.
\bibitem{Nielsen2} Daniel E. Cohen: Combinatorial group theory a topological approach. London Mathematical Society Student Texts. 14. Cambridge University Press, Cambridge.
\bibitem{Jensen} C. A. Jensen: Contractibility of fixed point sets of auter space. Topology and its Applications. \hyperlink{https://doi.org/10.1016/S0166-8641(01)00074-8}{https://doi.org/10.1016/S0166-8641(01)00074-8}. 
\bibitem{Collins} D. Collins: Palindromic automorphism of free groups, in: Combinatorial and Geometric group theory, in: London Mat. Soc. Lecture Note Ser., Vol 204. Cambridge Univ. Press, Cambridge, 1995, PP. 63-72.
\bibitem{KG} Valeriy G. Bardakov, Krishnendu Gongopadhyay, Mahender Singh: Palindromic automorphisms of free groups. Journal of Algebra. 438, (2015). PP. 260-282.
\end{thebibliography}
\end{document}